\documentclass[a4paper,12pt]{amsart}
\usepackage{amsmath,amssymb,inputenc,euscript,graphicx,psfrag,xcolor}

\title[Schr\"odinger equations on Damek-Ricci spaces]
{Schr\"odinger equations\\
on Damek-Ricci spaces}

\author{Jean-Philippe Anker}
 
\address{Universit\'e d'Orl\'eans \& CNRS,
F\'ed\'eration Denis Poisson (FR 2964) \& Laboratoire MAPMO (UMR 6628),
B\^atiment de math\'ematiques -- Route de Chartres,
B.P. 6759 -- 45067 Orl\'eans cedex 2 -- France}

\email{anker@univ-orleans.fr}

\author{Vittoria Pierfelice}

\address{Universit\'e d'Orl\'eans \& CNRS,
F\'ed\'eration Denis Poisson (FR 2964) \& Laboratoire MAPMO (UMR 6628),
B\^atiment de math\'ematiques -- Route de Chartres,
B.P. 6759 -- 45067 Orl\'eans cedex 2 -- France}

\email{vittoria.pierfelice@univ-orleans.fr}

\author{Maria Vallarino}

\address{Universita' di Milano-Bicocca,
Dipartimento di Matematica e Applicazioni,
Via Cozzi 53 -- 20125 Milano (MI) -- Italia}

\email{maria.vallarino@unimib.it}

\thanks{This work was mostly carried out
while the third author was a CNRS postdoc
at the {\it F\'ed\'eration Denis Poisson\/} Orl\'eans-Tours}

\date{\today}



\newtheorem{lemma}{Lemma}[section]
\newtheorem{theorem}[lemma]{Theorem}
\newtheorem{prop}[lemma]{Proposition}
\newtheorem{corollary}[lemma]{Corollary}

\newtheorem{remark}[lemma]{Remark}
\newtheorem{definition}[lemma]{Definition}

\newcommand{\la}{\lambda}

\newcommand{\di}{\,\textrm{d}}

\begin{document}

\begin{abstract}   
In this paper we consider the Laplace-Beltrami operator $\Delta$ on  Damek-Ricci spaces and derive pointwise estimates for the kernel of $e^{\tau\Delta}$, when $\tau\in\mathbb C^*$ with $\mathrm{Re}\,\tau\ge0$. When $\tau\in i\mathbb R^*$, we obtain in particular pointwise estimates of the Schr\"odinger kernel associated with $\Delta$.
We then prove Strichartz estimates for the Schr\"odinger equation, for a family
of admissible pairs which is larger than in the Euclidean case. This extends the results obtained by Anker and Pierfelice \cite{AP} on real hyperbolic spaces.
As a further application, we study the dispersive properties of the Schr\"odinger equation associated with a distinguished Laplacian on  Damek-Ricci spaces, showing that in this case the standard $L^1\to L^{\infty}$ estimate fails while suitable weighted Strichartz estimates hold.
\end{abstract}

\maketitle

\section{Introduction}

The study of the dispersive properties of many evolution equations of mathematical physics, including the Schr\"odin\-ger and heat equation
on $\mathbb{R}^{n}$, is of fundamental importance. 
Indeed, dispersive estimates
represent the main tool in the study of several
linear and nonlinear problems. 
We recall some standard facts. Consider the homogeneous Cauchy problem for the linear  Schr\"odinger equation on $\mathbb{R}^n$, $n \geq 1$,
\begin{equation}
\begin{cases}
\;i\,\partial_tu(t,x)+\Delta u(t,x)=0\,\\
\;u(0,x)=f(x),\\
\end{cases}
\end{equation}
whose solution can be represented as
\begin{equation*}
u(t,x)
=e^{\hspace{.25mm}i\hspace{.25mm}t\hspace{.25mm}\Delta}f(x)=
\frac1{(4\pi i t)^{\frac n2}}\int_{\mathbb{R}^{n}}e^{-i\frac{|x-y|^2}{4t}}
f(y)dy \qquad  \forall t\neq 0\,.
\end{equation*} 
By the explicit representation of the kernel of $e^{it \Delta}$
one easily obtains the dispersive estimate
$$
\|e^{it\Delta}\|_{L^1(\mathbb{R}^n)\rightarrow L^\infty(\mathbb{R}^n)}
\lesssim 
|t|^{-\frac n2}\,\qquad\forall t\ne0\,.
$$
It is sufficient to get rid of $i$ in the kernel to obtain
a corresponding representation for the heat kernel of $e^{t \Delta}$,
and hence a similar dispersive estimate
$$
\|e^{t\Delta}\|_{L^1(\mathbb{R}^n)\rightarrow L^\infty(\mathbb{R}^n)}
\lesssim 
t^{-\frac n2}\qquad\forall t>0\,.
$$
It is well known that, starting from the dispersive estimates, it is possible to deduce other space-time estimates which are called Strichartz estimates.
The first such estimate was obtained by Strichartz himself in a special
case; then
Ginibre and Velo \cite{GV} obtained the complete
range of estimates with the exclusion of some
critical cases, the \emph{endpoint} cases, which were
finally proved by Keel and Tao \cite{KT}.
We recall that the modern theory of local and global well posedness 
for semilinear Schr\"odinger equations is based
essentially on these estimates.

In view of the important applications to nonlinear problems, many attempts have been made to study the dispersive properties for the corresponding equations on various Riemannian manifolds
(see e.g. \cite{AP, Ba, BCS, Bo, BGT, HTW, IS, P1, P2} among the others).

More precisely, dispersive and Strichartz estimates for the Schr\"odinger  equation on 
real hyperbolic spaces $\mathbb{H}^n$, which are
manifolds with constant negative curvature, have been stated by Banica, Anker and Pierfelice, Ionescu and Staffilani
(\cite{AP, Ba, BCS, IS, P1, P2}).
Here we are interested in extending these results to the more general context of \emph{Damek-Ricci spaces}, also known as
\emph{harmonic $NA$ groups} (\cite{ADY, Bog, CDKR1, CDKR2, D1, D2, DR1, DR2, R}).
As Riemannian manifolds, these solvable Lie groups include all symmetric spaces of noncompact type and rank one, but most of them are not symmetric, 
thus providing counterexemples
to the Lichnerowicz conjecture \cite{DR1}. 

We briefly recall the definition of the spaces.
Let $\mathfrak{n}=\mathfrak{v}\oplus\mathfrak{z}$ be an Heisenberg-type algebra and let $N$ be the connected and simply connected Lie group associated to $\mathfrak{n}$ (see Section~2 for the details). Let $S$ be the one-dimensional extension of $N$ obtained by making $A=\mathbb{R}^+$ act 
on $N$ by homogeneous dilations. We denote by $Q$ the homogeneous dimension of $N$ and by $n$ the dimension of $S$. Let $H$ denote a vector in $\mathfrak{a}$ acting on
$\mathfrak{n}$ with eigenvalues $1/2$ and (possibly) $1$; we extend the inner product 
on $\mathfrak{n}$ to the algebra $\mathfrak{s}=\mathfrak{n}\oplus\mathfrak{a}$, by requiring $\mathfrak{n}$
and $\mathfrak{a}$ to be orthogonal and $H$ to be a unit vector. We denote by $d$ the left invariant distance on $S$ associated with the
Riemannian metric on $S$ which agrees with the inner product 
on $\mathfrak{s}$ at the identity. 
The Riemannian manifold $(S,d)$ is usually referred to as \emph{Damek-Ricci space}.

Note that $S$ is nonunimodular in general; denote by
$\lambda$ and $\rho$ the left and right Haar measures on $S$, respectively. It is well known that the spaces $(S,d,\lambda)$ 
and $(S,d,\rho)$ are of \emph{exponential growth}.
In particular, the two following Laplacians on $S$ 
have been the object of investigation\,:
\begin{itemize}
\item[(i)] The Laplace-Beltrami operator $\Delta_S$ associated 
with the Riemannian metric $d$.
The operator $-\Delta_S$ is left invariant, it is essentially selfadjoint on $L^2(S,\lambda)$ and its spectrum is the half line $[Q^2/4,\infty)$.
\item[(ii)] The left invariant Laplacian $\mathcal L=\sum_{i=0}^{n-1}{X}_i^2$, where $X_0,...X_{n-1}$ are left invariant vector fields such that at the identity $X_0=H$, $\{X_1,\ldots,X_{m_{\mathfrak{v}}}\}$ is an orthonormal basis of $\mathfrak{v}$ 
and $\{X_{m_{\mathfrak{v}}+1},\ldots,X_{n-1}\}$ is an orthonormal basis of $\mathfrak{z}$. The operator $-\mathcal L$ is essentially selfadjoint on $L^2(S,\rho)$ and its spectrum is $[0,\infty)$. 
\end{itemize}
Considerable effort has been produced to study the so-called 
$L^p$--functional calculus for the operators $-\Delta_S$ and $-\mathcal L$. 
It turned out that if $p\neq 2$, then $-\Delta_S$ possesses
a $L^p$ \emph{holomorphic} functional calculus 
\cite{CS}, whereas $-\mathcal L$ admits a
$L^p$ functional calculus \emph{of Mihlin-H\"ormander type} 
\cite{A2, CGHM}. 
This interesting dichotomy between the two operators 
motivated many authors to study both of them in the context 
of real hyperbolic spaces, in noncompact symmetric spaces of rank one or, 
more generally, in Damek-Ricci spaces and in noncompact symmetric spaces of arbitrary rank \cite{An, AJ, A2, CGM1,CGM2, CGM3, HS, I1, I2, MT, MV, V}.

In this paper we study the dispersive properties of the Schr\"odinger equations on $S$ associated with 
both the Laplacians $\Delta_S$ and $\mathcal L$. 

To this end, in Section \ref{pointwise} we start by proving pointwise estimates of the kernel of the more general operator $e^{\tau\Delta_S}$, for $\tau\in\mathbb C^*$ with $\mathrm{Re}\,\tau\ge0$. These can be thought as estimates of the heat kernel of the Laplacian $\Delta_S$ in complex time and are obtained using the inversion formula for the Abel transform. Similar results were proved in \cite{DM, M} on real hyperbolic spaces. 

In the special case when $\mathrm{Re}\,\tau=0$ this gives pointwise estimates of the Schr\"odinger kernel of $e^{it\Delta_S}$, for $t\in\mathbb R^*$. These imply the following \emph{dispersive estimates:}  
$$
\| e^{it\Delta_S} \|_{      L^{\tilde q'}(S,\lambda)\rightarrow    L^q(S,\lambda)  }
\lesssim 
\begin{cases}
|t|^{-\max\{\frac12-\frac1q, \frac12-\frac1{\tilde{q}} \}n}\, \qquad {\rm{if~}} 0<|t|<1\\
|t|^{-\frac32}\, \qquad {\rm{if~}}  |t|\geq 1\,,
\end{cases}
$$
for all $q,\tilde{q}\in (2,\infty]$. As a consequence, we deduce that the solution $u$ of the nonhomogeneous Cauchy problem 
\begin{equation}\label{LB}
\begin{cases}
\;i\,\partial_tu(t,x)+\Delta_Su(t,x)=F(t,x)\,\\
\;u(0,x)=f(x)\,, \quad x \in S,  
\end{cases}
\end{equation}
satisfies the following \emph{Strichartz estimates}
\begin{equation}\label{str}
\|u\|_{L^p(\mathbb R;\,L^q(S,\lambda))}\lesssim \|f\|_{L^2(S,\lambda)}+\|F\|_{L^{\tilde{p}'}(\mathbb R;\, L^{\tilde{q}'}(S,\lambda))}
\,,
\end{equation}
for all couples $(\frac1p,\frac1q)$ and $(\frac1{\tilde{p}},\frac1{\tilde{q}})$ which lie in the \emph{admissible triangle}
$$
T_n=\Big\{
\Big(\frac1p,\frac1q\Big)\in \Bigl(0,\frac12\Bigr]\times\Bigl(0,\frac12\Bigr) :~
\frac2p+\frac nq\ge\frac n2\Big\}\cup\Big\{\Bigl(0,\frac12\Bigr)\Big\}\,.
$$
Note that the set $T_n$ of admissible pairs for $S$ is much wider that the admissible interval $I_n$ for $\mathbb R^n$ which is just the lower edge of the triangle.  This phenomenon was already observed by Anker and Pierfelice for real hyperbolic spaces \cite{AP} and here is generalized 
to Damek-Ricci spaces. 
\smallskip

As an application of the estimates \eqref{str}, we study the dispersive properties of the {\emph{Schr\"odinger equation associated with $\mathcal L$}:}
\begin{equation}\label{distinguishedL}
\begin{cases}
\;i\,\partial_tu(t,x) + \mathcal L u(t,x)=F(t,x)\,\\
\;u(0,x)=f(x)\,,\quad x\in S\,.
\end{cases}
\end{equation}
In this case we prove that there is no dispersive $L^1-L^{\infty}$ 
estimate for the solution of the homogeneous Cauchy problem. Beside this, we are able to show that the solution 
of the nonhomogeneous Cauchy problem \eqref{distinguishedL} satisfies suitable weighted Strichartz 
estimates for couples $(\frac1p,\frac1q)$ and $(\frac1{\tilde{p}},\frac1{\tilde{q}})$ in the admissible triangle $T_n$. More precisely, we obtain this result as an application of the Strichartz estimates proved for the equation associated with $\Delta_S$ using a special relationship between the two Laplacians (see (\ref{relationship})). 

\bigskip
Note that, in the 
particular case of real hyperbolic spaces, D.~M\"uller and C.~Thiele 
found a similar lack of the dispersive effect for the wave equation 
associated with $\mathcal L$ and they suggested that Strichartz estimates 
shall not hold in that case
\cite[Remark 7.2]{MT}.  

\section{Damek-Ricci spaces}
In this section we recall the definition of $H$-type groups, describe their Damek-Ricci extensions, and recall the main results of spherical analysis on these spaces. For the details see \cite{ADY, A2, CDKR1, CDKR2, D1, D2, DR1, DR2}.
\smallskip

Let $\mathfrak{n}$ be a Lie algebra equipped  with an inner product $\langle\cdot,\cdot\rangle$ and denote by $|\cdot |$ the corre\-spon\-ding norm. Let $\mathfrak{v}$ and $\mathfrak{z}$ be complementary orthogonal subspaces of $\mathfrak{n}$ such that $[\mathfrak{n},\mathfrak{z} ]=\{0\}$ and $[\mathfrak{n},\mathfrak{n}]\subseteq \mathfrak{z}$.
According to Kaplan \cite{K}, the algebra $\mathfrak{n}$ is of $H$-type if for every $Z$ in $\mathfrak{z}$ of unit length the map $J_Z:\mathfrak{v}\to \mathfrak{v}$, defined by
$$\langle J_ZX,Y\rangle\,=\,\langle Z,[X,Y]\rangle\qquad\forall X, Y\in \mathfrak{v}\,,$$
is orthogonal. The connected and simply connected Lie group $N$ associated to $\mathfrak{n}$ is called an $H$-type group. We identify $N$ with its Lie algebra $\mathfrak{n}$ via the exponential map
\begin{eqnarray*}
\mathfrak{v}\times\mathfrak{z} &\to& N\\
(X,Z)&\mapsto& \exp(X+Z)\,.
\end{eqnarray*}
The product law in $N$ is
$$(X,Z)(X',Z')=\Bigl(X+X',Z+Z'+\frac12\,[X,X']\Bigr)\qquad\forall X,\,X'\in \mathfrak{v}\quad\forall Z,\,Z'\in\mathfrak{z}\,.$$
The group $N$ is a two-step nilpotent group, hence unimodular, with Haar measure ${\mathrm{d}} X \,{\mathrm{d}} Z$.
We define the following dilations on $N$:
$$\delta_a(X,Z)=(a^{1/2}X,aZ)\qquad\forall (X,Z)\in N \quad\forall a\in\mathbb{R}^+\,.$$

Set $Q=(m+2k)/2$, where $m$ and $k$ denote the dimensions of $\mathfrak{v}$ and $\mathfrak{z}$, respectively. The number $Q$ is called the homogeneous dimension of $N$.

Let $S$ be the one-dimensional extension of $N$ obtained by making $A=\mathbb{R}^+$ act on $N$ by homogeneous dilations. We shall denote by $n$ the dimension $m+k+1$ of $S$. Let $H$ denote a vector in $\mathfrak{a}$ acting on $\mathfrak{n}$ with eigenvalues $1/2$ and (possibly) $1$; we extend the inner product on $\mathfrak{n}$ to the algebra $\mathfrak{s}=\mathfrak{n}\oplus\mathfrak{a}$, by requiring $\mathfrak{n}$ and $\mathfrak{a}$ to be orthogonal and $H$ to be a unit vector. 
The map
\begin{eqnarray*}
\mathfrak{v}\times\mathfrak{z}\times\mathbb{R}^+ &\to& S\\
(X,Z,a)&\mapsto& \exp(X+Z)\exp(\log a \,H)
\end{eqnarray*}
gives global coordinates on $S$. The product in $S$ is given by the rule
$$(X,Z,a)(X',Z',a')=\Bigl(X+a^{1/2}X',Z+a\,Z'+ \frac12\,a^{1/2}[X,X'],a\,a'\Bigr)$$
for all $(X,Z,a),\,(X',Z',a')\in S$. The
group $S$ is nonunimodular: the right and left Haar measures on
$S$ are  given by 
$${\mathrm{d}}\rho(X,Z,a)=a^{-1}\,{\mathrm{d}} X\,{\mathrm{d}} Z\,{\mathrm{d}} a
\qquad{\textrm{and}}\qquad
{\mathrm{d}}\lambda(X,Z,a)=a^{-(Q+1)}\,{\mathrm{d}} X\,{\mathrm{d}} Z\,{\mathrm{d}} a \,.$$  
Then the modular function is $\delta(X,Z,a)=a^{-Q}$. For $p\in [1,\infty)$ we
denote by $L^p(S,\lambda)$ and $L^p(S,\rho)$ the spaces of all measurable 
functions $f$ such that $\int_S|f|^p \,{\mathrm{d}}\lambda <\infty$ and $\int_S|f|^p \,{\mathrm{d}}\rho <\infty$, respectively.

We equip $S$ with the left invariant Riemannian metric which agrees with the inner product on $\mathfrak{s}$ at the identity. From \cite[formula (2.18)]{ADY}, for all $(X,Z,a)$ in~$S$,
\begin{equation}\label{distanza}
\cosh ^2\left(\frac{r(X,Z,a)}2\right)=\left(\frac{a^{1/2}+a^{-1/2}}2+\frac18\,a^{-1/2}|X|^2\right)^2+\frac14\,a^{-1}|Z|^2\,,
\end{equation}
where $r(X,Z,a)$ denotes the distance of the point $(X,Z,a)$ from the identity.

We denote by $\Delta_S$ the Laplace-Beltrami operator associated with 
this Remannian structure on $S$.

A radial function on $S$ is a function that depends only on the distance from the identity. If $f$ is radial, then by \cite[formula (1.16)]{ADY}, 
$$
\int_Sf\,{\mathrm{d}}\lambda =\int_0^\infty f(r)\,A(r)\,{\mathrm{d}} r\,,
$$
where 
\begin{equation}\label{stimaA}
A(r)=2^{m+k}\sinh ^{m+k}\left(\frac r2\right)\cosh^k\left(\frac r2\right)\qquad\forall r\in\mathbb{R}^+\,.
\end{equation}
A radial function $\phi$ is spherical if it is an eigenfunction of $\Delta_S$ and $\phi(e)=1$. For $s$ in $\mathbb{C}$, 
let $\phi_s$ be the spherical function with eigenvalue $-\big(s ^2+Q^2/4\big)$, as in \cite[formula (2.6)]{ADY}. The spherical Fourier transform ${\mathcal H} f$ of an integrable radial function $f$ on $S$ is defined by the formula
$${\mathcal H} f(s)=\int_S \phi_s\,f\,{\mathrm{d}}\lambda  \,.$$
For ``nice'' radial functions $f$ on $S$, an inversion formula and a Plancherel formula hold:  
$$f(x)=c_S\int_{0}^{\infty}{\mathcal H} f (s)\,\phi_s(x)\,| {\mathbf{c}} (s) |^{-2} \,{\mathrm{d}} s\qquad \forall x\in S\,,$$
and
$$\int_S|f|^2\,{\mathrm{d}}\lambda  =c_S\int_0^{\infty}|{\mathcal H} f (s)|^2\,|{\mathbf{c}}(s)|^{-2}\,{\mathrm{d}} s\,,$$
where the constant $c_S$ depends only on $m$ and $k$, and ${\mathbf{c}}$ denotes the Harish-Chandra function.
 
Let ${\mathcal A}$ denote the Abel transform and let ${\mathcal F}$ denote the Fourier transform on the real line, defined by 
$${\mathcal F} g(s)=\int_{-\infty}^{+\infty}g(r)\,{\mathrm{e}}^{-isr}\,{\mathrm{d}} r\,,$$ 
for each integrable function $g$ on~$\mathbb{R}$. It is well known that ${\mathcal H}={\mathcal F}\circ {\mathcal A}$, hence ${\mathcal H}^{-1}={\mathcal A}^{-1}\circ {\mathcal F}^{-1}$. For later use, we recall the inversion formula for the Abel transform \cite[formula (2.24)]{ADY}. We define the differential operators ${\mathcal D}_1$ and ${\mathcal D}_2$  on the real line by
\begin{equation}\label{D1D2}
{\mathcal D}_1=\,-\frac1{\sinh r}\,\frac\partial{\partial r}\,,\qquad {\mathcal D}_2=\,-\frac1{\sinh(r/2)}\,\frac\partial{\partial r} \,.
\end{equation} 
If $k$ is even, then
\begin{equation}\label{inv1}
{\mathcal A}^{-1}f(r)=a_S^e\,{\mathcal D}_1^{k/2}{\mathcal D}_2^{m/2}f (r)\,,
\end{equation}
where $a_S^e=2^{-(3m+k)/2}\pi^{-(m+k)/2}$, while if $k$ is odd, then
\begin{equation}\label{inv2}
{\mathcal A}^{-1}f(r)=a_S^o\int_r^{\infty}{\mathcal D}_1^{(k+1)/2}{\mathcal D}_2^{m/2}f (s) \,{\mathrm{d}}\nu(s)\,,
\end{equation}
where $a_S^o= 2^{-(3m+k)/2}\pi^{-n/2}$ and ${\mathrm{d}}\nu(s)=(\cosh s-\cosh r)^{-1/2}\sinh s \,{\mathrm{d}} s$.

\bigskip
Let $
\mathcal L= \sum_{i=0}^{n-1} X_i^2\,
$ be the left invariant Laplacian defined in the Introduction.  
There is a special relationship between $\mathcal L$ and $\Delta_S$. Indeed, denote by ${\Delta}_Q$ the shifted operator 
$\Delta_S+{Q^2}/4$; then by \cite[Proposition~2]{A2},
\begin{equation}\label{relationship}
\delta^{-1/2}{(-\mathcal L)}\,\delta^{1/2}f=-\Delta_Q f
\end{equation}
for all smooth compactly supported radial functions $f$ on $S$.  
The spectra of $-{\Delta}_Q$ on $L^2(S,\lambda)$ and $-\mathcal L$ on $L^2(S,\rho)$ are both $[0,+\infty)$. Let $E_{{\Delta} _Q}$ and $E_{\mathcal L}$ be the spectral resolution of the identity for which 
$$-{\Delta} _Q=\int_0^{+\infty}s\, \,{\mathrm{d}} E_{{\Delta} _Q}(s)
\qquad{{\textrm{and}}}\qquad-\mathcal L=\int_0^{+\infty}s\, \,{\mathrm{d}} E_{\mathcal L}(s)\,.$$
For each bounded measurable function $m$ on $\mathbb{R}^+$ the operators $m(-{\Delta} _Q)$ and $m(-\mathcal L)$, spectrally defined by  
$$m(-{\Delta} _Q)=\int_0^{+\infty}m(s) \,{\mathrm{d}} E_{{\Delta} _Q}(s)\qquad{{\textrm{and}}}\qquad m(-\mathcal L)=\int_0^{+\infty}m(s) \,{\mathrm{d}} E_{\mathcal L}(s)\,,$$
are bounded on $L^2(S,\lambda)$ and $L^2(S,\rho)$ respectively. By (\ref{relationship}) and the spectral theorem, 
\begin{equation}\label{relazionemolt}
\delta^{-1/2}m(-\mathcal L)\,\delta^{1/2}f=m(-{\Delta} _Q)f\,,
\end{equation}
for smooth compactly supported radial functions $f$ on $S$.
Let ${k_{m(-\mathcal L)}}$ and ${k_{m(-{\Delta}_Q)}}$ denote the convolution kernels of $m(-\mathcal L)$ and $m(-{\Delta}_Q)$ respectively; then 
$$m(-{\Delta}_{Q}) f=f\ast{k_{m(-{\Delta}_Q)}}
\qquad{{\textrm{and}}}\qquad 
m(-\mathcal L) f=f\ast{k_{m(-\mathcal L)}}\qquad\forall f\in C^{\infty}_c(S)\,,$$
where $\ast$ denotes the convolution on $S$, defined by
\begin{equation*}
f\ast g(x)=\int_Sf(xy)\,g(y^{-1})\,{\mathrm{d}}\lambda  (y)=
\int_Sf(xy^{-1})\,g(y)\,{\mathrm{d}}\rho   (y)\,,
\end{equation*}
for all functions $f,g$ in $C_c(S)$ and $x$ in $S$.
Given a bounded measurable function $m$ on $\mathbb{R}^+$ the kernel 
${k_{m(-{\Delta}_Q)}}$ is radial and 
\begin{equation}\label{relazionenuclei}
{k_{m(-\mathcal L)}}=\delta^{1/2}\,
{k_{m(-{\Delta}_Q)}}\,.
\end{equation}
 Moreover, the spherical transform ${\mathcal H} k_{m(-{\Delta}_Q)}$ of $k_{m(-{\Delta}_Q)}$ is given by 
\begin{equation}\label{Hk}
{\mathcal H} k_{m(-{\Delta}_Q)}(s)=m(s^2)\qquad\forall s \in\mathbb{R}^+\,.
\end{equation}
For a proof of formula (\ref{Hk}) see \cite{ADY, A2}.

\section {Pointwise kernel estimates }\label{pointwise}
We consider the general operator $e^{\tau\Delta_S}$, where $\tau=|\tau|\,e^{i\theta}\in\mathbb C\setminus \{0\}$ and denote by $h_\tau $ its convolution kernel. Our aim is to find pointwise estimates of this kernel when $\mathrm{Re}\,\tau\ge0$. Notice that if $\tau\in\mathbb R^+$, then $h_\tau $ corresponds to the heat kernel and if $\tau=it\in i\mathbb R\setminus \{0\}$, then it corresponds to the Schr\"odinger kernel on Damek-Ricci spaces.

Notice that, for any  $\tau\in\mathbb C\setminus \{0\}$ with $\mathrm{Re}\,\tau\ge0$, we have $e^{\tau\Delta_S}=m_\tau (-\Delta_Q)$, where 
$m_\tau (v)=e^{-\frac{Q^2\tau}4-\tau v}$. Then by (\ref{Hk}), the spherical Fourier transform of $h_\tau $ is  
$$
\mathcal H h_\tau (s)=m_\tau (s^2)=e^{-\frac{Q^2\tau}4}\,e^{-\tau s^2}\,,
$$
and by applying the inverse Abel transform (\ref{inv1}) and (\ref{inv2}), we obtain the following formula for the kernel $h_\tau $:
\begin{equation}\label{steven}
h_\tau (r)=\begin{cases}
C\,(|\tau|e^{i\theta})^{-\frac12}\,e^{-\frac{Q^2\tau}4}\,
{\mathcal D}_1^{k/2}{\mathcal D}_2^{m/2}
\bigl(e^{\hspace{.25mm}-\frac{r^2}{4\tau}}\bigr)&{\rm{if}}\, k \,{\rm{even}}\,,\\
C\,(|\tau|e^{i\theta})^{-\frac12}\,e^{-\frac{Q^2\tau}4}\,
\int_r^{\infty}{\mathcal D}_1^{(k+1)/2}{\mathcal D}_2^{m/2}
\bigl(e^{\hspace{.25mm}-\frac{s^2}{4\tau}}\bigr)
\,{\mathrm{d}}\nu(s)&{\rm{if}}\, $k$ \,{\rm{odd}}\,,
\end{cases}
\end{equation}
where $\mathcal D_1=- \frac1{\sinh r}\frac\partial{\partial r}$ and 
$\mathcal D_2=-\frac1{\sinh (r/2)}\frac\partial{\partial r}$. 
We now prove a pointwise estimate of the kernel $h_\tau $.
\begin{prop}\label{PointwiseKernelEstimates}
There exists a positive constant \,$C\!$
\,such that, for every \,$\tau\in\mathbb C^*$ with $\mathrm{Re}\,\tau\ge0$ and \,for any $r\in\mathbb R^+$\;, we have
\begin{equation}\label{ker2}
|h_\tau (r)|\leq\,
\begin{cases}
C\,|\tau|^{-n/2}\,(1+r)^{\frac{n-1}2}\,e^{-\frac Q2r}\,e^{-\frac14\mathrm{Re}\,\{Q^2\tau+\frac{r^2}\tau\}}
&\text{if}\hspace{3mm}|\tau|\!\le\!1\!+\!r\,,\\
C\,|\tau|^{-3/2}\,(1+r)\,e^{-\frac Q2r}\,e^{-\frac14\mathrm{Re}\,\{Q^2\tau+\frac{r^2}\tau\}      }
&\text{if}\hspace{3mm}|\tau|\!>\!1\!+\!r\,.\\
\end{cases}
\end{equation}
\end{prop} 
\begin{proof}
We shall resume in part the analysis
carried out in \cite[Section 5]{ADY} \,for the heat kernel, in \cite[Proposition 3.1]{AP} and \cite{DM, M} in the case of real hyperbolic spaces. Following the same ideas of \cite[Corollary 5.21]{ADY} by induction we can prove that for any $p,\,q\in \mathbb{N}$ 
such that $p+q\geq 1$
\begin{equation}\label{derivativepq}
\begin{aligned}
\mathcal D_1^q\,\mathcal D_2^p\bigl(e^{-\frac{r^2}{4\tau}}\bigr)&
=e^{-\frac{r^2}{4\tau}}
\,\sum_{j=1}^{p+q}
\tau^{-j}\,a_j(r)\,,
\end{aligned}
\end{equation}
where $a_j$ are finite linear combinations of products $f_{p_1,q_1},...,f_{p_j,q_j}$ with $p_1+...+p_j=p,\,q_1+...+q_j=q$ and
$$
f_{p,q}(r)\asymp (1+r)\,e^{-(p/2+q)r}.
$$
Thus $a_j(r)=O\big((1+r)^j\,e^{-\frac{(p+2q)}2r}\big)$.

We first consider the case when $k$ is even. By (\ref{steven}) and 
(\ref{derivativepq}) we obtain that
\begin{equation*}
\begin{aligned}
|h_\tau (r)|&\lesssim|\tau|^{-1/2}\,e^{-\frac14\mathrm{Re}\,\{Q^2\tau+\frac{r^2}\tau\}}\,\sum_{j=1}^{(k+m)/2}|\tau|^{-j}\,(1+r)^j\,e^{-\frac{(m+2k)}4r}\\
&\lesssim|\tau|^{-1/2}\,e^{-\frac14\mathrm{Re}\,\{Q^2\tau+\frac{r^2}\tau\}}\,e^{-\frac Q2r}\,\Big[\frac{1+r}{|\tau|}+\Big(\frac{1+r}{|\tau|}\Big)^{(n-1)/2}  \Big]\,.
\end{aligned}
\end{equation*}
This easily implies the desired estimate in this case. 

Let now consider the case when $k$ is odd. By (\ref{steven}) and 
(\ref{derivativepq}) we obtain  
\begin{equation*}
\begin{aligned}
|h_\tau (r)|&\lesssim \sum_{j=1}^{(k+1+m)/2}|\tau|^{-j}\,
|\tau|^{-1/2}\,\int_r^{\infty}\di s
\,\frac{\sinh s}{\sqrt{\cosh s-\cosh r}}\,
(1+s)^j\times\\
&\times \,e^{-\frac{m+2(k+1)}4s}\,\,e^{-\frac14\mathrm{Re}\,\{Q^2\tau+\frac{s^2}\tau\}}\,.
\end{aligned}
\end{equation*}
Here and throughout the proof,
we make repeated use of the following elementary estimates\,:
\begin{equation}\label{eq4}
\sinh s\,\asymp\frac s{1+s}\,e^{\hskip.2mm s}\,,
\end{equation}
and
\begin{equation}\label{eq5}
\begin{aligned}
\cosh s-\cosh r\,
&=\,2\,\sinh\frac{s-r}2\,\sinh\frac{s+r}2 
\asymp\, \frac{s-r}{1+s-r}\,\frac s{1+s}\,e^{ s}
\end{aligned}
\end{equation}
or
\begin{equation}\label{eqnew}
\begin{aligned}
\cosh s-\cosh r\,
&\asymp\, 
 \begin{cases}
\frac{s^2\!-r^2}{1\,+\,r}\,e^{\hskip.2mm r}
&\text{if \,}r\!\le\!s\!\le\!r\!+\!1\,,\\
\qquad e^{\hskip.2mm s}
&\text{if \,}s\!\ge\!r\!+\!1\,.\\
\end{cases}
\end{aligned}
\end{equation}
By (\ref{eq4}) and (\ref{eq5}) we get
\begin{equation}
\begin{aligned}
|h_\tau (r)|&\lesssim |\tau|^{-1/2}\,e^{-\frac14\mathrm{Re}\,\{Q^2\tau\}    }\,
\int_r^{\infty}\di s
\sqrt{\frac{1+s-r}{s-r}}\,\sqrt{\frac{1+s}s}\,e^{-s/2}\,
\frac s{1+s}\,e^s\times\\
&\times\Big[\frac{1+s}{|\tau|}+\Big(\frac{1+s}{|\tau|}\Big)^{n/2}\Big]\,
e^{-\frac Q2s}\,e^{-\frac s2}\,e^{-\frac14\mathrm{Re}\,\{\frac{s^2}\tau\}}.
\end{aligned}
\end{equation}
After performing the change of variables \,$s\!=\!r\!+\!u$, we obtain 
$$
|h_\tau (r)|\lesssim |\tau|^{-1/2}\,e^{-\frac14\mathrm{Re}\,\{Q^2\tau\}    }\,
\int_0^{\infty}\di u  \sqrt{\frac{1+u}u}\,\sqrt{\frac{r+u}{1+r+u}}\,
e^{-\frac Q2(u+r)}\,e^{-\frac14\mathrm{Re}\,\{\frac{u^2+r^2+2ur}\tau\} }.
$$
Using the following inequalities
$$\frac{\sqrt{r+u\,}}{\sqrt{1+r+u\,}\,}\le1\;,
\quad
1\!+\!r\!+\!u\le(1\!+\!r)\,(1\!+\!u)\,,
$$
we obtain 
\begin{equation}\label{eq7}
|\,h_\tau (r)\,|\,\lesssim\,|\tau|^{-\frac12}\,e^{-\frac Q2r}\,e^{-\frac14\mathrm{Re}\,\{Q^2\tau+\frac{r^2}\tau\}    }
\,\Bigl\{\frac{1+r}{|\tau|}+\Bigl(\frac{1+r}{|\tau|}\Bigr)^{\frac n2}\Bigr\}\,.
\end{equation}
This allows us to obtain the desired estimate when  $|\tau|\!>\!1\!+\!r$\,.

If \,$|\tau|\!\le\!1\!+\!r$\,, in order to prove the estimate \eqref{ker2}, we need to reduce the power \,$\frac n2$  to \,$\frac{n-1}2$\,.
For this purpose, inside \eqref{steven}, let us rewrite 
\begin{equation*}
\mathcal D_1^{(k+1)/2}\,\mathcal D_2^{m/2}
(e^{\,-\frac{s^2}{4\tau}})= P(\tau,s)+R(\tau,s)\,,
\end{equation*}
obtaining
\begin{equation}\label{nucleo}
h_\tau (r)=
C\,(|\tau|e^{i\theta})^{-\frac12}\,e^{-\frac{Q^2\tau}4}\,
{\int_{\,r}^{+\infty}}\di s\,
\frac{\sinh s}{\sqrt{\cosh s\,-\,\cosh r\,}\,}
\,\left[P(\tau,s) + R(\tau,s) \right],
\end{equation}
where
$$P(\tau,s)= C\,\tau^{-\frac{(k+1+m)}2+1}\,s^{\frac{(k+1+m)}2-1}\,\Big(-\frac1{\sinh s}\Big)^{(k+1)/2}\,\Big(-\frac1{\sinh s/2}\Big)^{m/2}\,
\,\frac\partial{\partial s}\bigl(e^{\,-\frac{s^2}{4\tau}}\bigr)$$
and
$R(\tau,s)=\sum_{j=1}^{(k+1+m)/2-1}\tau^{-j}a_j(s)\,e^{\,-\frac{s^2}{4\tau}}$. 
Arguing as above, we can estimate the second term in the \eqref{nucleo} as
\begin{equation}\label{Resto}
\begin{aligned}
&\Big|  |\tau|^{-\frac12}\,e^{-\frac{Q^2\tau}4}\,\int_r^{\infty}\di s\,\frac{\sinh s}{\sqrt{\cosh s\,-\,\cosh r\,}\,} R(\tau,s)\Big|\\&\lesssim |\tau|^{-1/2}\,e^{-\frac14\mathrm{Re}\{{Q^2\tau}\}}\,e^{-\frac Q2r}\,\Big(\frac{1+r}{|\tau|}\Big)^{n/2-1}e^{-\frac14\mathrm{Re}\frac{r^2}\tau}\,.
\end{aligned}
\end{equation}
Hence, it remains to consider the integral
\begin{equation*}
I(\tau,r)={\int_{\,r}^{+\infty}}\di s\,
\frac{\sinh s}{\sqrt{\cosh s\,-\,\cosh r\,}\,}
\,P(\tau,s)\,,
\end{equation*}
\,when \,$|\tau| \le1+r$\,. Let us write
$$I(\tau,r)\,=\;I_1(\tau,r)\,+\;I_2(\tau,r)\,, 
$$
according to the following splitting
$$
\int_{\,r}^{+\infty}
=\int_{\,r}^{\sqrt{r^2+|\tau|\,}}
+\int_{\sqrt{r^2+|\tau|\,}}^{+\infty}.$$
To treat the first integral $I_1$, we differentiate
\,$\frac\partial{\partial s}\bigl(e^{ -\frac{s^2}{4\tau}}\!
\bigr)= -\frac s{2\tau}\,e^{ -\frac{s^2}{4\tau}}$
\,and use the estimates \eqref{eq4}, \eqref{eqnew}
together with the fact that \,$s$ is in $[\,r,r+1\,]$ obtaining
$$
\begin{aligned}
|I_1(\tau,r)|\,
&\lesssim\,|\tau|^{-\frac n2}\,(1+r)^{\frac{n-1}2}\,e^{-\frac Q2r}\,e^{-\frac14\mathrm{Re}\{\frac{r^2}\tau\}}
{\int_{\;r}^{\sqrt{r^2\hskip-.2mm+|\tau|\,}}}\hskip-1mm
\di s\,\frac s{\sqrt{s^2\hskip-.2mm-r^2\,}\,}\,\\
&=\;|\tau|^{-\frac n2+\frac12}\,(1+r)^{\frac{n-1}2}\,e^{-\frac Q2r}\,e^{-\frac14\mathrm{Re}\{\frac{r^2}\tau\}}\,.
\end{aligned}
$$
By integrating by parts in the integral $I_2$, we get
\begin{equation*}
I_2(\tau,r) 
=g(\tau,r)+J(\tau,r) \,,
\end{equation*}
where 
\begin{equation*}
\begin{aligned}
&g(\tau,r)\\
&=\tau^{-\frac n2+1}
\frac{\sinh s}{\sqrt{\cosh s\,-\,\cosh r\,}\,}\,
s^{\frac n2-1}
\Bigl(-\frac1{\sinh s}\Bigr)^{(k+1)/2}\,\Bigl(-\frac1{\sinh s/2}\Bigr)^{m/2}\,
 \bigl(e^{ -\frac{s^2}{4\tau}}\bigr)
\Big|_{s=  \sqrt{r^2+|\tau|} }^{s=+\infty}
\end{aligned}
\end{equation*}
and
\begin{equation*}
\begin{aligned}
J(\tau,r) &= -\tau^{-\frac n2+1}\int_{\sqrt{r^2+|\tau|}}^{+\infty}\di s \,e^{\,-\frac{s^2}{4\tau}}\times\\
&\phantom{= -\tau^{-\frac n2+1}}\times\frac\partial {\partial s}\Big[
\frac{\sinh s}{\sqrt{\cosh s\,-\,\cosh r\,}\,}
s^{\frac n2-1}
\Bigl(-\frac1{\sinh s}\Bigr)^{(k+1)/2}\,\Bigl(-\frac1{\sinh s/2}\Bigr)^{m/2}\,\Big] .  
\end{aligned}
\end{equation*}
We first estimate the boundary term $g(\tau,r)$, in the following way
\begin{equation*}
\begin{aligned}
|g(\tau,r)|&\lesssim\,|\tau|^{-\frac n2+1}\,(1+\sqrt{r^2+|\tau|})^{\frac n2-1}\,e^{-\frac Q2r}\,e^{-\frac14\mathrm{Re}\{\frac{r^2}\tau\} }\\
&\lesssim |\tau|^{-\frac n2+\frac12}\,(1+r)^{\frac{n-1}2}\,e^{-\frac Q2r}\,e^{-\frac14\mathrm{Re}\{\frac{r^2}\tau\}}\qquad\forall |\tau|\leq 1+r.
\end{aligned}
\end{equation*}
Then to estimate the integral term $J$ we write it
$$
J(\tau,r)=J_1(\tau,r)+J_2(\tau,r)\,,
$$
according to 
$$
\int_{ \sqrt{r^2+|\tau|}  }^{+\infty}
=\int_{\sqrt{r^2+|\tau|\,}}^{r+1}
+\int_{r+1}^{+\infty}.$$
By computing the derivative which appears inside the first integral $J_1$ 
and using the elementary estimate $\frac{s\coth s\,-\,1}{\sinh s}\asymp s\,e^{-s}$, we obtain
\begin{equation*}
\begin{aligned}
|J_1(\tau,r)|&\lesssim   |\tau|^{-\frac n2+1}\,\int_{\sqrt{r^2+|\tau|}}^{r+1}\di s 
\,(1+s)^{n/2-2}\,s\,e^{-\frac Q2s}\,e^{-\frac14\mathrm{Re}\{\frac{s^2}\tau\}}\,
\Big\{
e^{-\frac r2}\,e^{\frac s2}\,
\Big(\frac{1+r}{s^2-r^2}\Big)^{1/2} \\
&+e^{-\frac32r}\,e^{\frac32s}\,\,
\Big(\frac{1+r}{s^2-r^2}\Big)^{3/2}
\Big\}\\
&\lesssim|\tau|^{-\frac n2+1}\,e^{-\frac14\mathrm{Re}\{\frac{r^2}\tau\}}\,(1+r)^{\frac{n-3}2}\,e^{-\frac Q2r}\,
\int_{\sqrt{r^2+|\tau|}}^{r+1}\di s 
\Big\{ \frac s{\sqrt{s^2-r^2}}+(1+r)\,\frac s{(s^2-r^2)^{3/2}}\Big\}\\
&\lesssim|\tau|^{-\frac n2+\frac12}\,\,e^{-\frac14\mathrm{Re}\{\frac{r^2}\tau\}}\,e^{-\frac Q2r}\,(1+r)^{\frac{n-1}2}
\qquad \forall |\tau|\leq 1+r\,.
\end{aligned}
\end{equation*}
We can estimate the second  integral $J_2$  as above, changing variable $s=r+u$ and getting
\begin{equation*}
\begin{aligned}
|J_2(\tau,r)|&\lesssim   |\tau|^{-\frac n2+1}\,\int_{r+1}^{\infty}\di s\, 
 (1+s)^{n/2-2}\,s\,e^{-\frac Q2s}\,e^{-\frac14\mathrm{Re}\{\frac{s^2}\tau\} }\\
&\lesssim|\tau|^{-\frac n2+1}\,e^{-\frac14\mathrm{Re}\{\frac{r^2}\tau\}}\, \int_1^\infty\di u \,(1+r+u)^{n/2-2}\,(u+r)\,e^{-\frac Q2(u+r)}\\
&\lesssim|\tau|^{-\frac n2+1}\,\, e^{-\frac14\mathrm{Re}\{\frac{r^2}\tau\}}\,(1+r)^{n/2-1}\,e^{-\frac Q2r}\\
&\lesssim|\tau|^{-\frac n2+\frac12}\,\,e^{-\frac14\mathrm{Re}\{\frac{r^2}\tau\}}\,(1+r)^{\frac{n-1}2}\,e^{-\frac Q2r}
\qquad \forall |\tau|\leq 1+r\,.
\end{aligned}
\end{equation*}
We have, thus, proved that $$I(\tau,r)=O\big(|\tau|^{-\frac n2+\frac12}\,(1+r)^{\frac{n-1}2}\,e^{-\frac Q2r}\,e^{-\frac14\mathrm{Re}\{\frac{r^2}\tau\}  }\big).$$ Then, the first term in \eqref{nucleo} is estimated by
\begin{equation}\label{Itaur}
\big|  |\tau|^{-\frac12}\,e^{-\frac{Q^2}4\tau}\,I(\tau,r) \big|\lesssim 
|\tau|^{-\frac n2}\,(1+r)^{\frac{n-1}2}\,e^{-\frac Q2r}\,e^{-\frac14\mathrm{Re}\{\frac{r^2}\tau\}  }\,,
\end{equation}
which, combined with \eqref{Resto}, allows to conclude the proof.
\end{proof}

\begin{remark}
Notice that, in the case when $\tau\in\mathbb R^+$ estimates from below for the heat kernel $h_\tau $ were proved in \cite[Theorem 3.1]{DM} for real hyperbolic spaces and \cite[Theorem 5.9]{ADY} for Damek-Ricci spaces.  In Lemma \ref{stbelow}, we shall prove an estimate from below for the kernel $h_\tau $ when $\tau\in i\mathbb R\setminus \{0\}$ in a suitable region of the space. 

\end{remark}

As a consequence of the previous pointwise bounds of the kernel $h_\tau $, 
we can estimate its norm in the weak Lorentz spaces.
\begin{corollary}\label{LqKernelEstimate}
Let \,$2\!<\!q\!<\!\infty$ \,and \,$1\!\le\!\alpha\!\le\!\infty$\,.
Then there exists a positive constant \,$C\!$ \,such that
the following kernel estimate holds
\begin{equation}\label{ker}
\|h_\tau \|_{L^{q,\alpha}(S,\lambda)}\leq\,\,
\begin{cases}
C\,|\tau|^{-n/2}\,e^{-\frac{Q^2}4\mathrm{Re}\tau    }
&\text{if}\hspace{3mm}0\!<\!|\tau|\!\le\!1\,,\\
C\,|\tau|^{-3/2}\,e^{-\frac{Q^2}4\mathrm{Re}\tau    }
&\text{if}\hspace{6mm}|\tau|\!>\!1\,.\\
\end{cases}
\end{equation}
\end{corollary}
\begin{proof}
It suffices to argue as in \cite[Lemma 3.2]{AP}. We recall that
the Lorentz spaces $L^{q, \alpha}(S,\lambda)$ are the spaces of functions $f$ such that :\\
if $1\leq \alpha<\infty$, then 
$$\|f\|_{L^{q,\alpha}(S,\lambda)}
=\Bigl[\,{\int_{\,0}^{+\infty}}\!\frac{ds}s\,
\{s^{1/q}f^*(s)\}^\alpha\,\Bigr]^{1/\alpha} 
=\Bigl[\,{\int_{\,0}^{+\infty}}\!dr\,
\frac{V'(r)}{V(r)}\,\{V(r)^{1/q}f(r)\}^\alpha\,\Bigr]^{1/\alpha}<\infty,$$
or, if $\alpha = \infty$, then
$$\|f\|_{L^{q,\infty}(S,\lambda)}
=\sup_{\,s>0}s^{1/q}f^*(s)
=\sup_{\,r>0}V(r)^{1/q}f(r)<\infty,$$
 where  the decreasing function $f^*$ is the rearrangement of $f$.\\
Notice that, if $f$ is a radial decreasing function, then 
$f^*=f\circ V^{-1},$ where
\begin{equation*}
\begin{aligned}
V(r)&=C \!{\int_{0}^{r}}\di s\, \sinh^{m+k}(s/2)\,\cosh^k(s/2)\\
&\asymp 
\begin{cases}
& r^n\quad \text{as} \,\, r \to 0 \\
& e^{Qr}\quad \text{as} \,\, r\to+\infty.
\end{cases}
\end{aligned}
\end{equation*}
Replacing $f$ by $h_\tau $  and using the kernel estimate \eqref{ker2}, we conclude the proof. \end{proof}

\section{Schr\"odinger equation on Damek-Ricci spaces}
Consider the homogeneous Cauchy problem for the linear Schr\"odinger equation 
associated with the Laplace-Beltrami operator on a Damek-Ricci space $S$
\begin{equation}\label{HS}
\begin{cases}
\;i\,\partial_tu(t,x)+\Delta_Su(t,x)=0\,\\
\;u(0,x)=f(x)\,,\quad x\in S\,,\\
\end{cases}
\end{equation}
whose solution is given by
\begin{equation*}
u(t,x)
=e^{\hspace{.25mm}i\hspace{.25mm}t\hspace{.25mm}\Delta_S}f(x)
=f\!*\!s_{t}(x)\,,
\end{equation*}
where we denote by $s_t$ the kernel $h_{it}$ for any $t\in\mathbb R\setminus \{0\}$. Our aim is to study the dispersive properties of $e^{it\Delta_S}$, and to do so we follow the strategy used in \cite{AP}, where the authors applied the Kunze-Stein phenomenon. Notice that, for general functions on Damek-Ricci spaces,  
this phenomenon is 
known to be false 
\cite{Co1, CF, Li, Lo}. To solve this difficulty, we define 
suitable spaces of radial functions on $S$ which have a ``nice'' 
convolution property.
\begin{definition}
For $s$ in $[2,\infty)$, we define $\mathcal A_s$ as the space of all radial 
function $\kappa$ on $S$ such that
$$  
\int_0^{\infty}\di \,|r\kappa(r)|^{s/2}\,\phi_0(r)\,A(r)<\infty \,,
$$
where $A(r)$ is the radial density of the left measure, introduced in \eqref{stimaA}. Given $\kappa$ in $\mathcal A_s$, set 
$$\|\kappa\|_{\mathcal A_s}=\Big(\int_0^{\infty}\di r\,|\kappa(r)|^{s/2}\,\phi_0(r)\,A(r)\,  \Big)^{2/s}\,.$$

For $s=\infty$ we denote by $\mathcal A_{\infty}$ the space of $L^{\infty}(S,\lambda)$ radial functions on $S$ and by $\|\cdot\|_{\mathcal A_{\infty}}$ the $L^{\infty}$-norm.
\end{definition}
We observe that $\mathcal A_s$ may be identified with the weighted space 
$L^{s/2}\big((0,\infty),\phi_0(r)\,A(r)\di r \big)$.

\begin{theorem}\label{convolution}
For any $q\in [2,\infty]$ we have that
$$
L^{q'}(S,\lambda)\ast \mathcal A_q\subset L^q(S,\lambda).
$$
More precisely, there exists a constant $C_q$ such that  for every function $f$ in $L^{q'}(S,\lambda)$ and $\kappa$ in $\mathcal A_q$ 
$$
\|f\ast \kappa\|_{L^q(S,\lambda)}\leq C_q \,\|\kappa\|_{\mathcal A_q}\, \|f\|_{L^{q'}(S,\lambda)}.
$$
\end{theorem}
\begin{proof}
The case when $q=2$ follows by \cite[Theorem (3.3)]{ADY}. When $q=\infty$, taking $f$ in $L^1(S,\lambda)$ 
and $\kappa$ in $\mathcal A_{\infty}$, we have that for every $x$ in $S$
$$
|f\ast \kappa(x) |\leq \int_S |f(y)| \,|\kappa(y^{-1}x)|\,\di\lambda (y)\leq 
\|k\|_{\infty}\,\|f\|_{L^1(S,\lambda)}=\|k\|_{\mathcal A_{L^{\infty}}}\,\|f\|_{L^1(S,\lambda)}.
$$
By interpolating between the case $q=2$ and $q=\infty$ we obtain that
\begin{equation}\label{conv}
[L^2(S,\lambda),L^1(S,\lambda)]_{\theta}\,\ast\, [\mathcal{A}_2,\mathcal{A}_{\infty}]_{\theta}\subset [L^2(S,\lambda),L^{\infty}(S,\lambda)]_{\theta}=L^q(S,\lambda)
\end{equation}
where $1/q=(1-\theta)/2$, with $\theta\in (0,1)$ (see \cite[Theorem 5.1.1]{BL}). Moreover, by \cite[Theorem 5.1.1]{BL}
$$[L^2(S,\lambda),L^1(S,\lambda)]_{\theta}=L^{q'}(S,\lambda)$$ and 
$$
[L^1\big((0,\infty),\phi_0\,A\di r\big),L^\infty\big((0,\infty),\phi_0\,A\di r \big)]_{\theta}=L^{q/2}\big((0,\infty),\phi_0\,A \di r \big).
$$
We then have $[\mathcal A_2,\mathcal A_{\infty}]_{\theta}=\mathcal A_q$, which combined with (\ref{conv}) implies the theorem.
\end{proof}
 
Let us turn to $L^q\to L^{\tilde{q}'}$ dispersive properties of the propagator $e^{\hspace{.25mm}\hspace{.25mm}it\hspace{.25mm}\Delta_S}$
on $S$.

\begin{theorem}\label{dispersive}
Let \,$2\!<\!q,\tilde q\!\le\!\infty$\,.
Then there exists a positive constant \,$C\!$ 
\,such that, for all $t\in\mathbb R\setminus \{0\}$, the following dispersive estimates hold\;
\begin{equation*}
\|\,e^{\hspace{.25mm}\hspace{.25mm}it\hspace{.25mm}\Delta_S}\,
\|_{L^{\tilde q'}(S,\lambda)\!\to L^q(S,\lambda)}\le\,\,
\begin{cases}
C\,|t|^{-\max\hspace{.25mm}\{\frac12-\frac1q,\frac12-\frac1{\tilde q}\}\,n}\, 
&\text{if}\hspace{3mm}0\!<\!|t|\!\leq\!1\,,\\
C\,|t|^{-\frac32}\, 
&\text{if}\hspace{6mm}|t|\!>\!1\,.\\
\end{cases}
\end{equation*}
\end{theorem}
\begin{proof}
For $0<|t|\leq 1$, by applying Corollary \ref{LqKernelEstimate}, we obtain
\begin{equation*}
\begin{cases}
\;\|\,e^{\hspace{.25mm}\hspace{.25mm}it\hspace{.25mm}\Delta_S}\,\|
_{L^1(S,\lambda)\to L^q(S,\lambda)}
=\|s_{t}\|_{L^q(S,\lambda)}\le C\;|t|^{-\frac n2}\, 
\quad\forall\;q\!>\!2\,,\\
\;\|\,e^{\hspace{.25mm}\hspace{.25mm}it\hspace{.25mm}\Delta_S}\,\|
_{L^{q'}(S,\lambda)\!\to L^{\infty}(S,\lambda)}\!
=\|s_{t}\|_{L^q(S,\lambda)}\le C\;|t|^{-\frac n2}\,
\quad\forall\;q\!>\!2\,,\\
\;\|\,e^{\hspace{.25mm}\hspace{.25mm}it\hspace{.25mm}\Delta_S}\,\|
_{L^2(S,\lambda)\to L^2(S,\lambda)}=1 \,.\\
\end{cases}
\end{equation*}
By interpolating the previous estimates, we deduce the desired result. 

When $|t|>1$, we apply Corollary \ref{LqKernelEstimate} and Theorem \ref{convolution} obtaining 
\begin{equation}\label{eitD}
\begin{cases}
\;\|\,e^{\hspace{.25mm}\hspace{.25mm}it\hspace{.25mm}\Delta_S}\,\|
_{L^1(S,\lambda)\to L^q(S,\lambda)}
=\|s_{t}\|_{L^q(S,\lambda)}\le C\;|t|^{-\frac32}\, 
\quad\forall\;q\!>\!2\,,\\
\;\|\,e^{\hspace{.25mm}\hspace{.25mm}it\hspace{.25mm}\Delta_S}\,\|
_{L^{q'}(S,\lambda)\!\to L^{\infty}(S,\lambda)}\!
=\|s_{t}\|_{L^q(S,\lambda)}\le C\;|t|^{-\frac32}\,
\quad\forall\;q\!>\!2\,,\\
\;\|\,e^{\hspace{.25mm}\hspace{.25mm}it\hspace{.25mm}\Delta_S}\,\|
_{L^{q'}(S,\lambda)\!\to L^q(S,\lambda)}\!
\le C_q\,\|s_{t}\|_{\mathcal A_q} 
\quad\forall\;q\!>\!2\,.\\
\end{cases}
\end{equation}
For any $2<q\leq \infty$, we have to estimate the $\mathcal A_q$-norm of the kernel $s_t$. To do so, we use Proposition \ref{PointwiseKernelEstimates}, formula  \eqref{stimaA} and the following inequality  (see \cite[Lemma 1]{A2})
$$
\phi_0(r)\leq C\,(1+r)\,e^{-\frac Q2r}\,,
$$
from which we deduce that $s_{t}$ lies in the space $\mathcal A_q$ and 
$$
\|s_{t}\|_{\mathcal A_q}\lesssim \,\,|t|^{-3/2}\qquad \forall |t|>1\, .
$$
Using the previous estimate in \eqref{eitD} and applying  interpolation, we conclude the proof when $|t|$ is large. 

\end{proof}
 
Finally, by combining dispersive estimates (proved in Theorem \ref{dispersive}) with the classical $TT^*$ method in the same way of \cite[Theorem 3.6]{AP}, 
we deduce the Strichartz estimates for a large family of admissible pairs.

\begin{theorem}\label{1}
Consider  the Cauchy Problem for the linear Schr\"odinger equation
\begin{equation*}
 \begin{cases}
  & i \partial_t u(t,x) + \Delta_S u(t,x) =F(t,x) \\
  & u(0,x)=f(x),\; \,x \in S.\\
 \end{cases}
\end{equation*}
If $(\frac1p,\frac1q)$ and $(\frac1{\tilde{p}},\frac1{\tilde{q}})$ lie in the admissible triangle 
\begin{equation}\label{Tn}
T_n=\Big\{
\Big(\frac1p,\frac1q\Big)\in \Bigl(0,\frac12\Bigr]\times \Bigl(0,\frac12\Bigr) :~
\frac2{p}  +\frac n{q}\geq \frac n2
\Big\} \cup 
\Big\{ \Bigl(0, \frac12 \Bigr)  \Big\},
\end{equation}
 then 
the solution 
$$u(t,x)=e^{it{\Delta_S}} f(x) + \int_0^t {\rm{\di}} s\, e^{i(t-s){\Delta_S}}F(s,x)$$
satisfies the following Strichartz estimates 
\begin{equation*}
\|u\|_{L^p(\mathbb R;L^q(S,\lambda))}\lesssim \|f\|_{L^2(S,\lambda)}+\|F\|_{L^{\tilde{p}'}(\mathbb R; L^{\tilde{q}'}(S,\lambda))}\,.
\end{equation*}
\end{theorem}
\begin{remark}
The applications to well-posedness and scattering theory for the NLS obtained in \cite[Section 4, 5]{AP} on real hyperbolic spaces, can be easily generalized to all Damek-Ricci spaces. We omit the details.
\end{remark}

\section{An application: the Schr\"odinger equation  associated with $\mathcal L$}
We now consider the homogeneous Cauchy problem for the linear Schr\"odinger equation
on a Damek-Ricci space $S$ associated with the distinguished Laplacian $\mathcal L$
\begin{equation*}
 \begin{cases}
  & i \partial_t u (t,x)+ \mathcal L u (t,x)=0 \\
  & u(0,x)=f(x),\; \,x \in S,
 \end{cases}
\end{equation*}
whose solution is given by
$$
u(t,x)=f\ast \sigma_t(x)\,,
$$
where $\sigma_t$ is the convolution kernel of the 
operator 
$e^{it\mathcal L}$. 

It is interesting to observe that $L^1-L^{\infty}$ dispersive estimate  for the Schr\"odinger equation 
associated with the Laplacian $\mathcal L$ does not hold. To show this fact, we first prove that the kernel $\sigma_t$ is not in $L^{\infty}$. More precisely, we estimate from below the kernel $s_{t}$ and we use the relationship (\ref{relazionenuclei}) between kernels of multipliers of $-\Delta_S$ and $-\mathcal L$, 
\begin{equation}\label{sigmat} 
\sigma_t=\delta^{1/2}\,e^{i\frac{Q^2t}4}\,s_{t},
\end{equation}
where $\delta$ is the modular function on $S$.

\begin{lemma}\label{stbelow}
For every $t$ in $\mathbb R\setminus\{0\}$ there exist positive constants $K$ and $c$, with $c>1$, such that
$$
|s_{t}(r)|\geq K\,|t|^{-n/2}\,r^{\frac{n-1}2}\,e^{-\frac Q2r}\qquad
\forall r> 1+c|t|.
$$
\end{lemma}
\begin{proof}
We suppose for simplicity $t>0$.
We first consider the case when $k$ is even. By the expression (\ref{steven}) of the kernel and the expansion (\ref{derivativepq}), we obtain
\begin{equation}
\begin{aligned}
s_{t}(r)&=C\,t^{-1/2}\,e^{i\frac{Q^2t}4}\,e^{i\frac{r^2}{4t}}\,
\sum_{j=1}^{(k+m)/2-1}t^{-j}\,a_j(r)\\
&+C\,t^{-1/2}\,t^{-(k+m)/2+1} \,r^{(k+m)/2-1}\,\Big(-\frac1{\sinh r}\Big)^{k/2-1}\,\Big(-\frac1{\sinh (r/2)}\Big)^{m/2}\,
\mathcal D_1(e^{i\frac{r^2}{4t}})\\
&+C\,t^{-1/2}\,t^{-(k+m)/2+1}\, r^{(k+m)/2-1} \, \Big(-\frac1{\sinh r}\Big)^{k/2}\,\Big(-\frac1{\sinh (r/2)}\Big)^{m/2-1}\,
\mathcal D_2(e^{i\frac{r^2}{4t}})\\
&= A(t,r)+B(t,r)\,,
\end{aligned}
\end{equation}
where $A(t,r)$ corresponds to the sum and $B(t,r)$ comes from the last two summands. By computing the derivatives which appear in the term $B(t,r)$, we obtain
$$
B(t,r)=
C\,t^{-1/2}\,t^{-(n-1)/2} \,\Big(\frac{ir}2\Big)^{(n-1)/2}\,\Big(-\frac1{\sinh r}\Big)^{k/2}\,\Big(-\frac1{\sinh (r/2)}\Big)^{m/2}\,e^{i\frac{r^2}{4t}}.
$$
As in the proof of Proposition \ref{PointwiseKernelEstimates}, we see that 
\begin{equation}\label{Atr}
  A(t,r)=O(t^{-n/2+1}\,r^{(n-1)/2-1}\,e^{-\frac Q2r} ) \qquad \forall r>1+t\,,
\end{equation}
and there exists a positive constant $C$ such that
\begin{equation}\label{Btr}
|B(t,r)|\geq C \,t^{-n/2}\,r^{(n-1)/2}\,e^{-\frac Q2r}\qquad \forall r>1+t\,.
\end{equation}
By (\ref{Atr}) and (\ref{Btr}), we deduce that there exists a sufficently large constant $c$ and a positive constant $K$ 
such that $|s_t(r)|\geq K \,t^{-n/2}\,r^{(n-1)/2}\,e^{-\frac Q2r}$, $\forall r>1+c\,t$, 
as required.

\bigskip
Suppose now that $k$ is odd. As before, by (\ref{steven}) and 
(\ref{derivativepq}), we can write
\begin{equation}
\begin{aligned}
s_{t}(r)&= \tilde A(t,r)+\tilde B(t,r)\,,
\end{aligned}
\end{equation}
where $\tilde A(t,r)=C\,t^{-1/2}\sum_{j=1}^{(k+1+m)/2-1}\int_r^{\infty}t^{-j}\,a_j(s)\,
e^{i\frac{s^2}{4t}}\di\nu(s)$ and 
\begin{equation}
\begin{aligned}
 \tilde B(t,r)=
C\,t^{-1/2}\,t^{-n/2+1}\int_r^{\infty}s^{n/2-1}\,
\Big(\frac1{\sinh s} \Big)^{(k+1)/2}\,\Big(\frac1{\sinh (s/2)} \Big)^{m/2}\,     
\frac\partial {\partial s}  \Big(e^{i\frac{s^2}{4t}}\Big)\di\nu(s) \,.
\end{aligned}
\end{equation}
As in the proof of Proposition \ref{PointwiseKernelEstimates}, we see that 
\begin{equation}\label{Jtr}
\tilde A(t,r)=O(t^{-(n-1)/2}\,r^{n/2-1}\,e^{-\frac Q2r} )\,.
\end{equation} 
Since 
$$
\di\nu(s)=\frac{\sinh s}{\sqrt  {\cosh s-\cosh r}}=\sinh s\,\Big(2\,\sinh \frac{s+r}2\,\sinh \frac{s-r}2 \Big)^{-1/2},
$$
the main term 
$\tilde B(t,r)$ can be written as 
\begin{equation*}
\begin{aligned}
\tilde B(t,r)&=C\,t^{-(n+1)/2}\,e^{i\frac{r^2}{4t}}\,
\int_r^{\infty} \di s\,s\,\Big(\sinh\frac{s+r}2\,\sinh\frac{s-r}2 \Big)^{-1/2}
\Big(\frac s{\sinh s}\Big)^{(k-1)/2}\times\\
&\times\Big(\frac s{\sinh s/2} \Big)^{m/2}
\,e^{i\frac{s^2-r^2}{4t}}\\
&=C\,t^{-(n+1)/2}\,e^{i\frac{r^2}{4t}}\,
\int_r^{\infty} \di s\,s\,f_r(s)\,e^{i\frac{s^2-r^2}{4t}}\,,
\end{aligned}
\end{equation*}
where $f_r(s)=\Big(\sinh \frac{s+r}2\,\sinh\frac{s-r}2\Big)^{-1/2}
\Big(\frac s{\sinh s}\Big)^{(k-1)/2}\,\Big(\frac s{\sinh s/2} \Big)^{m/2}$. 
By changing variables $u=\frac{s^2-r^2}{4t}$ the integral transforms into
$$
\tilde B(t,r)=C\,t^{-\frac{n-1}2}
\,e^{\,\frac{ir^2}{4t}}
\int_{\,0}^{+\infty}\di u\;
e^{\,iu}\,f_r(s(u)).
$$
Hence
\begin{equation*}
\begin{aligned}
|\tilde B(t,r)|&\geq |C|\,t^{-\frac{n-1}2}\,
{\rm{Im}}\,\Bigl\{
\int_{0}^{+\infty}\di u\,
e^{iu}\,f_r(s(u))  
\Bigr\}\,,
\end{aligned}
\end{equation*}
which can be split up in the following sum
$$
|C|\,t^{-\frac{n-1}2}\,
\sum_{j=0}^{+\infty}\,
\int_{\,2j\pi}^{(2j+1)\pi}\di u\,\sin u\,
\bigl\{f_r(s(u))-f_r(s(u+\pi))\bigr\}\,,
$$
which, since $u\,\longmapsto f_r(s(u))$ 
is a positive decreasing function, is estimated from below by 
$$ |C|\,t^{-\frac{n-1}2}
\int_{0}^{\pi}\di u\;\sin u\;
\bigl\{ f_r(s(u))-f_r(s(u+\pi)) \bigr\}\,.
$$ 
To estimate the last integral we write
$$
f_r(s(u))-f_r(s(u+\pi))
=\int_{\,0}^{\,\pi}\di v\,\{-f_r'(s(u\!+\!v)\}\,s'(u\!+\!v)\,.
$$
Notice that $s(u)=\sqrt{4tu+r^2}$, so that $s'(u)=\frac{2t}{s(u)}$. We now compute the derivative of $-f_r$ obtaining 
\begin{equation}\label{-fr'}
\begin{aligned}
-f_r'(s)
&=\frac14\,\Bigl(\sinh\frac{s+r}2\sinh\frac{s-r}2\Bigr)^{-\frac32}
\sinh s\,\Bigl(\frac s{\sinh s}\Bigr)^{(k-1)/2}\,
\Bigl(\frac s{\sinh s/2}\Bigr)^{m/2}\,
 \\
&+ \Bigl(\sinh\frac{s+r}2\sinh\frac{s-r}2\Bigr)^{-\frac12}
\Big[
\frac{k-1}2\Bigl(\frac s{\sinh s}\Bigr)^{(k-3)/2}
\frac{s\coth s-1}{\sinh s}\,\Bigl(\frac s{\sinh s/2}\Bigr)^{m/2}\\
&+\Bigl(\frac s{\sinh s}\Bigr)^{(k-1)/2}\,\frac m2\Bigl(\frac s{\sinh s/2}\Bigr)^{m/2-1}
\frac{\frac s2\coth (\frac s2)-1}{\sinh\frac s2}\,
\Big]\, .
\end{aligned}
\end{equation}
We now use in \eqref{-fr'} the elementary estimates 
$$
\sinh s\asymp e^s , \qquad 
\sinh (s/2)\asymp e^{s/2} ,
\qquad
s\coth s-1\asymp s,
$$
and 
$$
\sinh\frac{s+r}2\sinh\frac{s-r}2
\asymp\frac{s^2-r^2}s\,e^s\,,
$$
to obtain 
$$
-f_r'(s)\asymp (s^2-r^2)^{-1/2}\,s^{\frac{n-1}2}\,e^{-\frac Q2s}
\,\big[  (s^2-r^2)^{-1}\,s+2\big]\,.
$$
By replacing $s=s(u+v)=\sqrt{4t(u+v)+r^2}$, we get
\begin{equation}\label{-fr'2}
-f_r'\big(s(u+v)\big)\asymp \big(4t(u+v)\big)^{-1/2}\,s(u+v)^{\frac{n-1}2}\,e^{-\frac Q2s(u+v)}
\,\big[ \big(4t(u+v)\big)^{-1} \,s(u+v)+2\big]\,.
\end{equation}
Observe that, in the integral defining $\tilde B(t,r)$, we have $1\leq r\leq s$ and
$$
s
=r\,\sqrt{1\!+\hspace{-.2mm}4\,(u\!+\!v)\,\frac t{r^2}\,}\,.$$
Since $0<u+v<2\pi$ and $r>1+t$, we deduce
$$
s(u+v) \lesssim r\,\Big[ 1+\frac{4(u+v)\frac t{r^2}}2\Big] 
\lesssim r+2(u+v)\frac tr
\lesssim r+4\,\pi\,. 
$$
From (\ref{-fr'2}) and the previous estimates, we get 
\begin{equation}
\begin{aligned}
-\,f_r'(s(u\!+\!v))\,s'(u\!+\!v)
&\asymp \big(4t(u+v)\big)^{-1/2}\,r^{\frac{n-1}2}\,e^{-\frac Q2r}
\,\big[ \big(4t(u+v)\big)^{-1} \,r+2\big]\,\frac{2t}r\\
& \asymp \big(4t(u+v)\big)^{-1/2}\,r^{\frac{n-1}2}\,e^{-\frac Q2r}
\,\big[ \big(2(u+v)\big)^{-1} +\frac{4t}r\big] \\
&\asymp\,(u\!+\!v)^{-\frac32}\,t^{-\frac12}\,r^{\frac{n-1}2}\,e^{-\frac Q2r}
\qquad \forall r>1+t\,.
\end{aligned}
\end{equation}
Hence
$$
f_r(s(u))-f_r(s(u+\pi))
\asymp\,u^{-\frac12}\,t^{-\frac12}\,r^{\frac{n-1}2}\,e^{-\frac Q2r},
$$
so that we obtain
\begin{equation}\label{Itr}
|\tilde B(t,r)|\geq|C|\,t^{-\frac{n-1}2}
\int_{\,0}^{\,\pi}\di u\;\sin u\;
\bigl\{\,f_r(s(u))-f_r(s(u+\pi))\,\bigr\}
\geq C\,t^{-\frac n2}\,r^{\frac{n-1}2}\,e^{-\frac Q2r}.
\end{equation}
By (\ref{Jtr}) and (\ref{Itr}), we see that there exists a 
sufficiently large constant $c$ and a positive constant $K$ 
such that $|s_{t}(r)|\geq K\,t^{-\frac n2}\,r^{\frac{n-1}2}\,e^{-\frac Q2r}$  for all $r>1+ct$, as required.
\end{proof}
\begin{prop}
For every $t$ in $\mathbb R\setminus\{0\}$, the following hold:
\begin{itemize}
\item[(i)] the kernel $\sigma_t$ does not lie in $L^{\infty}(S,\rho)$;
\item[(ii)] the operator $e^{it\mathcal L}$ is not bounded from $L^1(S,\rho)$ to $L^{\infty}(S,\rho)$.
\end{itemize}
\end{prop}
\begin{proof}
Since $\sigma_t=\delta^{1/2}\, e^{i\frac{Q^2t}4} \,s_t$, from Lemma \ref{stbelow} we deduce that there exist constants $c>1$ 
and $K>0$ for which
$$
|\sigma_t(x)|\geq K\,|t|^{-n/2}\, \delta^{1/2}(x)\,r(x)^{\frac{n-1}2}\,
e^{-\frac Q2r(x)}
\qquad \forall r(x)>1+c|t|\,.
$$
Let $\Omega_t$ be the following region:
$$
\Omega_t=\{x=(X,Z,a)\in \mathfrak{v}\times\mathfrak{z}\times\mathbb R^+:~r(x)>1+c|t|, a<1, |(X,Z)|<1\}.
$$
By formula (\ref{distanza}), for any point $(X,Z,a)$ in 
$\Omega_t$, we have 
$$
e^{r(X,Z,a)}\asymp a^{-1}\qquad {\rm{and}} 
\qquad r(X,Z,a)\asymp \log(a^{-1}) .
$$
Hence for any point $(X,Z,a)$ in 
$\Omega_t$  
$$
|\sigma_t(X,Z,a)|\geq C\,a^{-Q/2}\,|t|^{-n/2}\,[\log(a^{-1})]^{\frac{n-1}2}\,
a^{Q/2}\geq C\,|t|^{-n/2}\,[\log(a^{-1})]^{\frac{n-1}2}.
$$
This shows that $\sigma_t$ is not in $L^{\infty}(S,\rho)$ and proves (i). 

Let now $\phi_n$ be a sequence of approximations of the identity, 
i.e. functions in $C_c^{\infty}(S)$ 
supported in the ball centred at the identity of radius $1/n$ such that 
$\|\phi_n\|_{L^1(S,\,\rho)}=1$, $0\leq \phi_n\leq 1$. 
Suppose that the operator $e^{it\mathcal L}$ is bounded from $L^1(S,\rho)$ to 
$L^{\infty}(S,\rho)$. Then there exists a constant $M$ such that 
$\|\phi_n\ast \sigma_t\|_{L^{\infty}(S,\rho)}\leq M$. Since $\phi_n\ast \sigma_t$ converges to $\sigma_t$ almost everywhere we deduce that $|\sigma_t|\leq M$ 
almost everywhere which contradicts (i). Thus the operator $e^{it\mathcal L}$ 
is not bounded from $L^1(S,\rho)$ to 
$L^{\infty}(S,\rho)$.
\end{proof}
Even if the $L^1-L^{\infty}$ dispersive estimate does not hold, we shall 
prove suitable weighted Strichartz estimates for the Schr\"odinger equation 
associated with the Laplacian $\mathcal L$. We shall deduce them from the 
Strichartz estimates which hold for the  Schr\"odinger equation associated 
with the Laplace-Beltrami operator. To do so, for any $q\in [2,\infty)$ 
we introduce the weight function $\delta_q$ defined by 
\begin{equation}\label{deltaq}
\delta_q=\delta^{1-q/2}=\delta^{q\,\big(\frac1q-\frac12\big)}.
\end{equation}
The weights $\delta_q$ are involved in a simple 
relationship between the $L^q$ norms of functions computed with respect to 
the right and left Haar measures.
\begin{lemma}\label{Lpnorms}
For any $q\in [2,\infty)$, the following hold:
\begin{itemize}
\item[(i)] $\|\delta^{-1/2}f\|_{L^q(S,\lambda)}=\|f\|_{L^q(S,\delta_q\rho)}$ for every $f$ in $L^q(S,\delta_q\rho)$;
\item[(ii)] $\|f\|_{L^q(S,\lambda)}=\|\delta^{1/2}f\|_{L^q(S,\delta_q\rho)}$ for every $f$ in $L^q(S,\lambda)$.
\end{itemize}
\end{lemma}
\begin{proof}
Take $f$ in $L^q(S,\delta_q\rho)$. We have that
\begin{equation}
\begin{aligned}
\|\delta^{-1/2}f\|^q_{L^q(S,\lambda)}&=
\int \delta^{-q/2}\,|f|^q\di\la\\
&=\int \delta^{-q/2}\,|f|^q\,\delta\di\rho\\
&=\int \delta^{q(1/q-1/2)}\,|f|^q\di\rho\\
&=\|f\|^q_{L^q(S,\delta_q\rho)}.
\end{aligned}
\end{equation}
This proves (i). The statement (ii) follows directly from (i).
\end{proof}
\begin{theorem}\label{translatedestimates}
Consider  the Cauchy Problem for the linear Schr\"odinger equation
\begin{equation*}
 \begin{cases}
  & i \partial_t u(t,x) + \mathcal L u(t,x) =F(t,x) \\
  & u(0,x)=f(x),\; \,x \in S.\\
 \end{cases}
\end{equation*}
For all $\big(\frac1p,\frac1q\big)$ and 
$\big(\frac1{\tilde{p}},\frac1{\tilde{q}}\big)$ 
in the admissible triangle $T_n$, the solution 
\begin{equation}\label{solution}
u(t,x)=e^{it{\mathcal L}} f(x) + \int_0^t \di s\, e^{i(t-s){\mathcal L}}F(s,x),
\end{equation}
satisfies the following weighted Strichartz estimates  
\begin{equation*}
\|u\|_{L^p(\mathbb R;\,L^q(S,\delta_q \,\rho))}\lesssim \, \|f\|_{L^2(S,\rho)}+
\|F\|_{L^{\tilde{p}'}(\mathbb R; \,L^{\tilde{q}'}(S,\delta_{\tilde{q}'}\,\rho))}\,.
\end{equation*}
\end{theorem}
\begin{proof}
By (\ref{sigmat}), we deduce that
\begin{equation}\label{ok}
u(t,x)=e^{i\frac{Q^2t}4}\,f\ast (\delta^{1/2}\,s_{t})(x)+
\int_0^t\di s \,e^{i\frac{Q^2(t-s)}4}\,\big[F\ast (\delta^{1/2}\,s_{(t-s)})\big](s,x).
\end{equation}
It is easy to see that for any functions $h,g$ on $S$ 
\begin{equation}\label{sposta}
h\ast (\delta^{1/2}g)=\delta^{1/2}\big[\big(\delta^{-1/2}h\big)\ast g\big].
\end{equation}
Applying (\ref{sposta}) in (\ref{ok}), we obtain 
\begin{equation}
\begin{aligned}
e^{-i\frac{Q^2t}4}\,\delta^{-1/2}\,u(t,x)&=\big(\delta^{-1/2}f \big)\ast s_{t}(x)+
\int_0^t\di s \,e^{-i\frac{Q^2s}4}\,\big[\delta^{-1/2}F\ast s_{(t-s)}\big](s,x)\,.
\end{aligned}
\end{equation}
Suppose now that $\big(\frac1p,\frac1q)$ and 
$\big(\frac1{\tilde{p}},\frac1{\tilde{q}}\big)$ lie 
in the admissible triangle $T_n$ introduced in (\ref{Tn}). 
By Lemma \ref{Lpnorms} and Theorem \ref{1}, we get
\begin{equation}
\begin{aligned}
\|u\|_{L^p(\mathbb R;L^q(S,\delta_q \,\rho))}&=
\|\delta^{-1/2}u\|_{L^p(\mathbb R;L^q(S,\lambda))}\\
&\lesssim\,\|\delta^{-1/2}f \|_{L^2(S,\lambda)}+\,\|\delta^{-1/2}F\|_{ L^{\tilde{p}'}(\mathbb R; L^{\tilde{q}'}(S,\lambda)) }\\
&=\,\|f\|_{L^2(S,\rho)}+\,\|F\|_{L^{\tilde{p}'}(\mathbb R;L^{\tilde{q}'}(S,\delta_{\tilde{q}'}\,\rho))},
\end{aligned}
\end{equation}
as required.
\end{proof}

\end{document}